\documentclass[a4paper,12pt]{article}

\usepackage{a4,amssymb,amsmath,amsthm,url,xcolor,enumerate,float,lineno}
\usepackage{bezier,amsfonts,amssymb,graphicx,amsthm,url}
\usepackage[utf8]{inputenc}
\usepackage{subcaption}
\usepackage{amsmath}
\usepackage{amsfonts}
\usepackage{amssymb,amsthm}
\usepackage{url}
\usepackage{xcolor}
\usepackage{tikzit}

\newtheorem{theorem}{Theorem}[section]
\newtheorem{definition}{Definition}[section]
\newtheorem{proposition}[theorem]{Proposition}

\newtheorem{lemma}[theorem]{Lemma}
\newtheorem*{theorem*}{Theorem}
\newtheorem{problem}{Problem}

\def\kaxxa{{\vcenter {\hrule height .2mm
\hbox{\vrule width .2mm height 2mm \kern 2mm
\vrule width .2mm} \hrule height .2mm}}}

\newcommand{\cay}[2]{\mathsf{Cay}\left(#1 ; #2\right)}

\tikzstyle{vertex}=[fill=black, draw=black, shape=circle, thick, scale=0.75]

\tikzstyle{edge}=[-, fill={rgb,255: red,0; green,128; blue,128}]
\tikzstyle{dir_edge}=[->]
\tikzstyle{blue_edge}=[-, draw=blue, ultra thick]
\tikzstyle{red_edge}=[-, draw=red, ultra thick]

\tikzset{every picture/.style={font issue=\footnotesize},
         font issue/.style={execute at begin picture={#1\selectfont}}
        }

\title{Flip colouring of graphs II}
\author{Xandru Mifsud \\ University of Malta}
\date{}

\DeclareGraphicsExtensions{.pdf,.png,.jpg}

\begin{document}


\maketitle

\begin{abstract}
We give results concerning two problems on the recently introduced \textit{flip colourings of graphs}. For positive integers $b, r$ with $b < r$, we say that a $b + r$ regular graph is a $(b,r)$-\textit{flip graph} if there exists a red/blue edge colouring such that the red degree of every vertex is $r$, the blue degree of every vertex is $b$, yet in the closed neighbourhood of every vertex there are more blue edges than red edges.

 We prove that for integers $b, r$ with $4 \leq b < r < b + 2 \left\lfloor\frac{b+2}{6}\right\rfloor^2$, small constructions of $(b,r)$-flip graphs on $\Theta(b+r)$ vertices are possible. Furthermore, we prove that there exist $k$-flip sequences $(a_1, \dots, a_k)$ where $k > 4$, such that $a_k$ can be arbitrarily large whilst $a_i$ is constant for $1 \leq i < \frac{k}{4}$. 
\end{abstract}
\section{Introduction}

Flip colourings of graphs were introduced in \cite{caro2023flip}, as yet another example of local versus global phenomena studied in graph theory, such as \cite{ABDULLAH20151, caro2018effect, FISHBURN1986165}. 

For positive integers $b, r$ with $b < r$, we say that a $b + r$ regular graph is a $(b,r)$-\textit{flip graph} if there exists an edge colouring $f: E(G) \rightarrow \{\mbox{\textsf{blue}, \textsf{red}}\}$ satisfying the following:
\begin{enumerate}[i.]
	\item The subgraphs induced by the \textsf{blue} and \textsf{red} edges are $b$ and $r$ regular respectively, resulting in a global majority ordering since $b < r$, where across the entire graph `\textsf{red}' wins against `\textsf{blue}'.
	\item On the other-hand, for every vertex $v$, the number of \textsf{blue} edges in the closed neighbourhood of $v$ is \textit{greater} than the number of \textsf{red} edges, resulting in a locally opposite majority ordering where locally `\textsf{blue}' wins against `\textsf{red}'.
\end{enumerate}

We term such a graph as a $(b, r)$-flip graph due to the local v. global majority flip they demonstrate. 

Several open problems concerning flip colourings were posed in \cite{caro2023flip}, some of which we consider here, whilst others have been studied in \cite{sheffield2025}. 

Before introducing the general problem for $k \geq 2$ colours, we establish some notational conventions. 

The open neighbourhood $N^G (v)$ is the set of neighbours of a vertex $v$ in a graph $G$, and the closed neighbourhood $N^G[v]$ is $N^G (v) \cup \{v\}$. 
Let $k \in \mathbb{N}$ and let $f \colon E(G) \rightarrow \{1, \dots, k\}$ be an edge-colouring of $G$. For $1 \leq j \leq k$,
\begin{enumerate}[i.]
	\item Given a subset $S$ of $V(G)$, $E^G_j (S)$ is the set of edges coloured $j$ in the vertex-induced subgraph of $G$ by $S$, and $e^G_j (S) = |E^G_j (S)|$.
	\item For a vertex $v$, let $e^G_j [v] = e^G_j \left(N^G[v]\right)$ and $e^G_j (v) = e^G_j \left(N^G(v)\right)$.
	\item For a vertex $v$, $\deg_j (v)$ is the number of edges incident to $v$ coloured $j$.
\end{enumerate}

When there is no ambiguity, we simplify our notation by removing any symbolic reference to the graph.

We are interested in the following problem: Given $k \geq 2$, a $d$-regular graph $G$ and an increasing positive integer sequence $a_1 < a_2 < \dots < a_k$ such that $d = \sum_{j=1}^k a_j$, does there exist an edge-colouring on $k$ colours such that:
\begin{enumerate}[i.]
	\item the set of edges coloured $j$ spans an $a_j$-regular subgraph of $G$, namely $\deg_j(v) = a_j$ for every $v\in V$,
	\item and for every vertex $v \in  V$,  $e_k[v] < e_{k-1}[v] < \ldots <  e_1[v]$.
\end{enumerate}   
 
If such an edge-colouring exists then $G$ is said to be an $(a_1,\ldots,a_k)$-flip graph, or more simply a $k$-flip graph, and $(a_1, \dots, a_k)$ is called a $k$-flip sequence of $G$. An illustrative example is given in Figure \ref{flip_4_3_smallest}.

\begin{figure}[ht!]
	\centering
	\ctikzfig{figures/flip_4_3_smallest}
	\caption{Smallest known $(3, 4)$-flip graph, with the subgraph induced by the closed neighbourhood of any vertex $v$ illustrated on the right. This is a $(3, 4)$-flip graph since: $\deg_{\mathsf{blue}}(v) = 3 < 4 = \deg_{\mathsf{red}}(v)$ but $e_{\mathsf{blue}}[v] = 7 > 6 = e_{\mathsf{red}}(v)$.}
	\label{flip_4_3_smallest}
\end{figure}

The case when $k = 2$ is fully characterised by the following theorem.

\begin{theorem}[Theorem 3.1, \cite{caro2023flip}]\label{RBClassification}
	Let $r, b \in \mathbb{N}$. If $3 \leq b < r \leq \binom{b+1}{2} - 1$ then there exists a $(b,r)$-flip graph, and both the upper and lower bounds are sharp.
\end{theorem}

Given a $k$-flip sequence $(a_1, \dots, a_k)$, a problem of interest is that of finding the smallest order $h(a_1, \dots, a_k)$ of a graph realising it. In the case when $k = 2$, the following theorem gives the best known upper bound on $h(b,r)$. 

\begin{theorem}[Corollary 3.6, \cite{caro2023flip}]\label{oldBound}
	Let $b, r \in \mathbb{N}$ such that $3 \leq b < r \le \binom{b+1}{2}-1$. Then,
$$ h(b,r) \leq 2 \left(r + b + 1 - \left\lfloor \dfrac{5 + \sqrt{1+8(r-b)}}{2} \right\rfloor\right) \left\lfloor \dfrac{5 + \sqrt{1+8(r-b)}}{2} \right\rfloor.
$$	
\end{theorem}

Smaller constructions for certain 2-flip sequences, such as $(3, 4)$, than those given by this upper bound are known. For $4 \leq b < r < b + 2 \left\lfloor\frac{b+2}{6}\right\rfloor^2$, we shall improve this to $h(b,r) \leq 16 \left(2 + \left\lfloor\frac{r}{2}\right\rfloor + \left\lfloor\frac{b+2}{2}\right\rfloor - 2\left\lfloor\frac{b+2}{6}\right\rfloor\right)$ 

For three colours, as in the case for two colours (Theorem \ref{RBClassification}), the largest colour-degree is quadratically bound in terms of the smallest.

\begin{theorem}[Theorem 4.1, \cite{caro2023flip}]\label{3colour_bound}
If $(a_1, a_2, a_3)$ is a $3$-flip sequence, then
$$a_3 < 2(a_1)^2.$$
\end{theorem}

However, for four or more colours, it is known that no such relationship exists, as highlighted in the following.

\begin{theorem}[Theorem 4.3, \cite{caro2023flip}] \label{arbitraryGapsThm_h_1}
	Let $k \in \mathbb{N}, k > 3$. Then there is some constant $m = m(k) \in \mathbb{N}$ such that for all $N \in \mathbb{N}$, there exists a $k$-flip sequence $\left(a_1, a_2, \dots, a_k\right)$ where $a_1 = m$ and $a_k > N$. 
\end{theorem}

Let $q(k)$ be the largest integer satisfying $q(k) < k$ such that there exists some $h(k) \in \mathbb{N}$ and for all $N \in \mathbb{N}$, there is a $k$-flip sequence $(a_1, \dots, a_k)$ where $a_{q(k)} = h(k)$ and $a_k > N$. In other words, $q(k)$ is the largest index in a $k$-flip sequence such that $a_{q(k)}$ can be some fixed value $h(k)$ but $q_k$ can be arbitrarily large. By Theorem \ref{arbitraryGapsThm_h_1} we establish that $q(k) \geq 1$ for all $k > 3$. 

Our contribution shall be that $$\max\left\{1, \left\lceil\frac{k}{4}\right\rceil - 1\right\} \leq q(k) < \begin{cases}
 	\ \frac{k}{3} & \mbox{if $k \equiv 0 \ (\!\!\!\!\!\!\mod 3)$} \\
 	\left\lceil \frac{k}{2}\right\rceil & \mbox{otherwise}
 \end{cases}$$ for $k > 3$.

Two foundational concepts employed heavily throughout this paper are \textit{Cayley graphs} and \textit{sum-free} sets, which we introduce next. 

Let $\Gamma$ be a group. We shall denote the identity of $\Gamma$ by $1_\Gamma$. All groups considered are assumed to be finite. We use the standard notation $\mathbb{Z}_n$ for the group of integers under addition modulo $n$.

Let $S$ be a subset of $\Gamma$ such that $S$ is inverse-closed and does not contain the identity. The \textit{Cayley graph} $\cay{\Gamma}{S}$ has vertex set $\Gamma$ and edge set $\left\{\{g, gs\}\colon s \in S, g \in \Gamma\right\}$. The set $S$ is termed as the \textit{connecting set}.
 
Let $\Gamma$ be an Abelian group and let $A, B \subseteq \Gamma$. The sum-set $A+B$ is the set $\{a+b\colon a \in A, b \in B\}$. By $2A$ we denote the set $A + A$ whilst by $A^{-1}$ we denote the set of inverses of $A$. We say that $A$ is sum-free if $2A \cap A = \emptyset$. 

In Section 2 we shall outline a number of properties of products and packings of edge-coloured graphs. Section 3 shall be dedicated to improving the upper bound on $h(b,r)$, while Section 4 presents lower and upper bounds on $q(k)$ for $k \geq 4$. In Section 5 we give some concluding remarks and further open problems to those in \cite{caro2023flip}.
\section{New edge-coloured graphs from old}

In this section we will briefly outline the toolset required, looking at a number of graph operations and how they affect edge-colourings. Namely we will consider the Cartesian and strong products of graphs, as well as the packing of graphs.

\subsection{Products of edge-coloured graphs}

We begin by recalling the definition of the strong product of two graphs $G$ and $H$, and in particular how an edge-colouring of $G \boxtimes H$ is inherited from edge-colourings of its factors. 

\begin{definition}[Strong product] The strong product $G \boxtimes H$ of two graphs $G$ and $H$ is the graph such that $V\left(G \boxtimes H\right) = V(G) \times V(H)$ and there is an edge $\left\{(u, v), (u^\prime, v^\prime)\right\}$ in $G \boxtimes H$ if, and only if, either $u = u^\prime$ and $v \sim v^\prime$ in $H$, or $v = v^\prime$ and $u \sim u^\prime$ in $G$, or $u \sim u^\prime$ in $G$ and $v \sim v^\prime$ in $H$.
\end{definition}

We extend the edge-colourings of $G$ and $H$ to an edge-colouring of $G \boxtimes H$ as follows. Consider the edge $e = \left\{(u, v), (u^\prime, v^\prime)\right\}$ in $G \boxtimes H$; if $u = u^\prime$ then $e$ inherits the colouring of the edge $\{v, v^\prime\}$ in $H$, otherwise if $u \neq u^\prime$ the colouring of the edge $\{v, v^\prime\}$ in $G$ is inherited. This colouring inheritance is illustrated in Figure \ref{SPDiagram}, with its properties summarised in Lemma \ref{SPColourLemma}.

\begin{lemma}\label{SPColourLemma}
	Let $G$ and $H$ be edge-coloured from $\{1, \dots, k\}$. Then in the coloured strong product $G \boxtimes H$, for any $1 \leq j \leq k$ and $(u, v) \in V(G \boxtimes H)$,
	\begin{enumerate}[i.]
		\item $\deg_j \big((u,v)\big) = \deg_j^H (v) + \deg_j^G (u) \left(1 + \deg^H (v)\right)$,
		\item $e_j \big[(u, v)\big] = e_j^H [v] \left(1 + \deg^G (u)\right) + e_j^G [u] \left(1 + \deg^H (v) + 2\sum\limits_{i = 1}^k e_i^H [v]\right)$.
	\end{enumerate}
\end{lemma}

\begin{figure}[h!]
\centering
\includegraphics[width=0.45\textwidth]{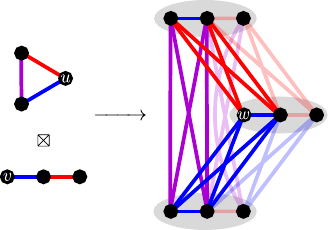}	
\caption{Illustration of Lemma \ref{SPColourLemma}, with the closed neighbourhood of $w = (u, v)$ in $K_3 \boxtimes P_3$ highlighted.}
\label{SPDiagram}
\end{figure}

Another useful product shall be Cartesian product, which we recall below.

\begin{definition}[Cartesian product] The Cartesian product $G \ \square \ H$ of the graphs $G$ and $H$ is the graph such that $V\left(G \ \square \ H\right) = V(G) \times V(H)$ and there is an edge $\left\{(u, v), (u^\prime, v^\prime)\right\}$ in $G \ \square \ H$ if, and only if, either $u = u^\prime$ and $v \sim v^\prime$ in $H$, or $v = v^\prime$ and $u \sim u^\prime$ in $G$.
\end{definition}

We extend the edge-colourings of $G$ and $H$ to an edge-colouring of $G \ \square \ H$ as follows. Consider the edge $e = \left\{(u, v), (u^\prime, v^\prime)\right\}$ in $G \ \square \ H$; if $u = u^\prime$ then $e$ inherits the colouring of the edge $\{v, v^\prime\}$ in $H$, otherwise if $v = v^\prime$ the colouring of the edge $\{u, u^\prime\}$ in $G$ is inherited. This colouring inheritance is illustrated in Figure \ref{CPDiagram}, with its properties summarised in Lemma \ref{CPColourLemma}.

\begin{figure}[h!]
\centering
\includegraphics[width=0.45\textwidth]{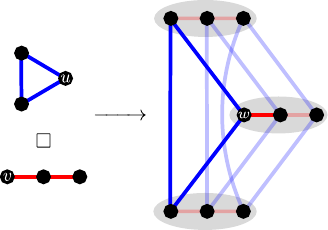}
\caption{Illustration of Lemma \ref{CPColourLemma}, with the closed neighbourhood of $w = (u, v)$ in $K_3 \ \square \ P_3$ highlighted.}
\label{CPDiagram}
\end{figure}

\begin{lemma}\label{CPColourLemma}
	Let $G$ and $H$ be edge-coloured from $\{1, \dots, k\}$. Then in the coloured Cartesian product $G \ \square \ H$, for any $1 \leq j \leq k$ and $(u, v) \in V(G \ \square \ H)$,
	\begin{enumerate}[i.]
		\item $\deg_j \big((u,v)\big) = \deg_j^G (u) + \deg_j^H (v)$,
		\item $e_j\left[(u,v)\right] = e_j^G[u] + e_j^H[v]$.
	\end{enumerate}
\end{lemma}

\subsection{Packing of edge-coloured graphs}

In this section we shall consider, in particular, the packing of edge-coloured Cayley graphs. We first formally define graph packing.

\begin{definition}[Packing]
	Two graphs $G$ and $H$ are said to pack if there exists bijections $g \colon V(G) \rightarrow \{1, \dots, n\}$ and $h \colon V(H) \rightarrow \{1, \dots, n\}$ such that the images of $E(G)$ and $E(H)$ under $g$ and $h$, respectively, do not intersect. 
	
	The packing of $G$ and $H$ is the graph with vertex set $\{1, \dots, n\}$ and edge set being union of the images of $E(G)$ and $E(H)$ under $g$ and $h$, respectively.
\end{definition}

An edge-colouring is naturally inherited by a packing of $G$ and $H$, by keeping the original colour of every single edge. Note that by the definition of a packing, the edge-colouring is well-defined.

Counting the colour-degree of every vertex in a packing of $G$ and $H$, in terms of the colour-degrees of $G$ and $H$, is straight-forward, however counting the coloured closed neighbourhood sizes is more difficult. In certain cases, such as when $G$ and $H$ are monochromatically coloured Cayley graphs on an Abelian group, we can do such counting. This is summarised in Proposition \ref{cayley_union_counting}. 

Let $S$ and $T$ be two disjoint inverse-closed subsets of $\Gamma$ not containing $1_\Gamma$. Consider the two Cayley graphs $G = \cay{\Gamma}{S}$ and $H = \cay{\Gamma}{T}$. Then the Cayley graph $\cay{\Gamma}{S \cup T}$ is a packing of $G$ and $H$. Cayley graphs enjoy a number of properties, most notably that they are vertex-transitive.

\begin{proposition}\label{cayley_union_counting}
	Let $\Gamma$ be an Abelian group and let $R, B$ be disjoint inverse-closed subsets of $\Gamma$ which do not contain $1_\Gamma$. Let $G = \cay{\Gamma}{B}$ and $H = \cay{\Gamma}{R}$ be monochromatically edge-coloured using colours $1$ and $2$, respectively. Then in $\cay{\Gamma}{B \cup R}$, for $v \in \Gamma$,
	\begin{enumerate}[i.]
		\item $\deg_1 (v) = \deg^G (v)$ and  $\deg_2 (v) = \deg^H (v)$,
		\item $e_1 [v] - e_2 [v] = \left(e_1^G [v] - e_2^H [v]\right) + \left(e_2^H \left(N^G (v)\right) - e_1^G \left(N^H (v)\right)\right)$, 
		\item furthermore, if $(R + B) \cap R = \emptyset$ and $e_1^G [v] > e_2^H [v]$, then $e_1 [v] > e_2 [v]$.
	\end{enumerate}
\end{proposition}

\begin{proof}
It suffices to consider a single vertex, say $1_\Gamma$, by virtue of the vertex-transitivity of Cayley graphs. Note that $B = N^G (1_\Gamma)$ and $R = N^H (1_\Gamma)$. More so, since $R$ and $B$ are disjoint, the edge-colouring of the union is well-defined and $N(1_\Gamma) = B \ \dot\cup \ R$. Clearly all the edges incident to $1_\Gamma$ are incident to $1_\Gamma$ in either $G$ or $H$, and therefore $\deg_1 (v) = \deg^G (v)$ and $\deg_2 (v) = \deg^H (v)$. We now count the number of edges coloured $1$ in the subgraph induced by $N(1_\Gamma)$. We have three cases for an edge $\{u, v\}$ coloured 1, 
\begin{enumerate}[i.]
\item Both $u$ and $v$ are in $B$, of which there are $e_1^G (B)$ such edges.
\item Both $u$ and $v$ are in $R$, of which there are $e_1^G \left(R\right)$ such edges.
\item The vertex $u$ is in $B$ and the vertex $v$ is in $R$. We show that the number of such edges is $2 e_2^H \left(B\right)$, \textit{i.e}. twice the number of edges coloured $2$ amongst the neighbours of $1_\Gamma$ in $G$.
\end{enumerate}

Since $\{u, v\}$ is coloured $1$ then it is an edge in the Cayley graph $G$. Therefore there is some $x \in B$ such that $u = xv$. Since $v^{-1} \in R$, then $\{u, x\}$ is an edge in $H$. Hence $x = u v^{-1}$ and since $\Gamma$ is Abelian and $u \in B$, $\{x, v^{-1}\}$ is an edge in $G$.

In other words, for every edge $\{u, x\}$ in $H$, where $u, x \in B$ and $u = xv$ for some $v \in R$, there are two edges $\{u, v\}$ and $\{x, v^{-1}\}$ in $G$ with one vertex in $R$ and one vertex in $B$. This counting argument is illustrated in Figure \ref{cayley_union_counting_fig}.

Hence, $e_1 [1_\Gamma] = e_1^G [1_\Gamma] + e_1^G \left(R\right) + 2 e_2^H \left(B\right)$. Repeating the argument for $e_2 [1_\Gamma]$ and subtracting, we get (ii) as required.

\begin{figure}[t!]
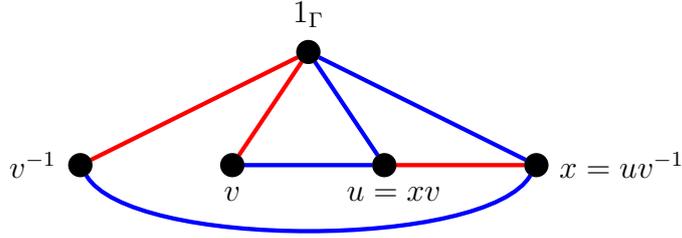

\centering
\ctikzfig{figures/cayley_union_colouring_counting}
\caption{Illustration of the counting argument in the proof of Proposition \ref{cayley_union_counting}, where the red edge $\{u, x\}$ between two blue neighbours of $1_\Gamma$ corresponds to two blue edges, each incident to a blue and red neighbour of $1_\Gamma$.}
\label{cayley_union_counting_fig}
\end{figure}

Now, suppose that $(R + B) \cap R = \emptyset$ and $e_1^G [v] > e_2^H [v]$. Then given any $u \in R$ and $v \in B$, $uv \notin R$ and therefore $\{u, uv\}$ is not an edge in the subgraph of $\cay{\Gamma}{B \cup R}$ induced by $R$. In other words, this subgraph has no edges coloured $1$ and therefore $e^G_1 \left(R\right) = 0$. Therefore (iii) follows from (ii). 
\end{proof}

\section{Bounding $h(b, r)$ through Cayley flip graphs}

If we can construct Cayley graphs satisfying the conditions in Proposition \ref{cayley_union_counting} (iii) and $|B| < |R|$, we can then construct a $\left(|B|, |R|\right)$-flip graph (which in particular turns out to be another Cayley graph). Note that the requirement $e_1 [v] > e_2 [v]$ necessitates that $|B| \geq 3$. Moreover, from the proof of Proposition \ref{cayley_union_counting} (iii) we have that in this case, $$\binom{|B| + 1}{2} \geq e_1^G [v] + e_2^H \left(N^G [v]\right) > e_2^H [v] \geq |R|$$ and therefore the constraints of Theorem \ref{RBClassification} are satisfied. In the proof of the following theorem we shall demonstrate not only that such constructions are feasible, but also that they can be done with a small number of vertices.

\begin{theorem}\label{newBound}
Let $b, r \in \mathbb{N}$ such that $4 \leq b < r < b + 2 \left\lfloor\frac{b+2}{6}\right\rfloor^2$. Then,
$$h(b, r) \leq 8 \lambda_{b,r}\left(2 + \bigg\lfloor\dfrac{r}{2}\bigg\rfloor + \bigg\lfloor\dfrac{b+2}{2}\bigg\rfloor - 2\bigg\lfloor\dfrac{b+2}{6}\bigg\rfloor\right)$$ where $\lambda_{b, r} = \max\{1, \left(b \mod 2\right) + \left( r \mod 2\right)\}$.
\end{theorem}

A key step in our proof will be the choice of suitably large subsets of $\mathbb{Z}_n$ which are both inverse-closed and sum-free. Sum-free sets in Abelian groups have been studied vastly and are of interest in additive combinatorics and number theory, see \cite{Alon_Kleitman_1990, GreenBen2005Ssia, TaoVan}. We therefore begin with the following useful lemma.

\begin{lemma}\label{sumsetLemma}
	Let $A_0, B_0$ be non-empty disjoint integer intervals of $\left(\frac{n}{8}, \frac{n}{4}\right)$ in $\mathbb{Z}_{n}$, such that $\max(A_0) < \min(B_0)$. Let $B_1 \subseteq B_0$ be an integer interval, $A = A_0 \cup A_0^{-1}$ and $B = B_0 \cup B_0^{-1} \cup 2B_1 \cup 2B_1^{-1}$. Then, $\left(A + B\right) \cap A = \emptyset$. Moreover if $n$ is even, $\left(A + \left\{\frac{n}{2}\right\}\right) \cap A = \emptyset$ and furthermore if $\min(B_1) \geq \frac{3n}{16}$ then $\left(\left\{\frac{n}{2}\right\} + B\right) \cap A = \emptyset$.
\end{lemma}

\begin{proof}
	Since $A_0$ and $B_0$ are disjoint integer intervals of $\left(\frac{n}{8}, \frac{n}{4}\right)$ such that $\max(A_0) < \min(B_0)$, then there exist integers $m, l, M, L$ such that $\frac{n}{8} < m < l < L < M < \frac{n}{4}$, $A_0 = \left[m, l\right]$ and $B_0 = \left[L, M\right]$. Observe that the sets $A_0$, $B_0$, $2B_1$ and their inverses are all disjoint and, in the case when $n$ is even, none include the involution $\frac{n}{2}$. 
	
	Now, $\min\left(A_0 + 2B_0^{-1}\right) = (m - 2M) \mod n > \max(A_0)$ and $$\max\left(A_0 + 2B_0^{-1}\right) = l - 2L < -l = \min\left(A_0^{-1}\right).$$ Hence $A_0 < A_0 + 2B_0^{-1} < A_0^{-1}$ and therefore since $B_1 \subseteq B_0$, $\left(A_0 + 2B_1^{-1}\right) \cap A = \emptyset$. 
	
	A similar argument follows for every possible summation, in order to obtain that $(A + B) \cap A = \emptyset$, as required.
	
	Now suppose that $n$ is even. Since $A + \left\{\frac{n}{2}\right\} \subseteq \left[\frac{n}{4}, \frac{3n}{8}\right] \cup \left[\frac{5n}{8}, \frac{3n}{4}\right]$ then $\left(A + \left\{\frac{n}{2}\right\}\right) \cap A = \emptyset$ as required.
	
	Also, $\left\{\frac{n}{2}\right\} + B_0 \cup B_0^{-1} \subseteq \left[\frac{n}{4}, \frac{3n}{8}\right] \cup \left[\frac{5n}{8}, \frac{3n}{4}\right]$. Moreover, $\min\left(\left\{\frac{n}{2}\right\} + 2B_1^{-1}\right) > 0$. If $\min(B_1) \geq \frac{3n}{16}$, then $$\max\left(\left\{\frac{n}{2}\right\} + 2B_1^{-1}\right) = \frac{n}{2} - 2\min(B_1) \leq \frac{n}{8}$$ and consequently $\left(\left\{\frac{n}{2}\right\} + 2B_1^{-1}\right) \cap A = \emptyset$. By a similar argument we obtain that $\left(\left\{\frac{n}{2}\right\} + 2B_1\right) \cap A = \emptyset$ and hence $\left(\left\{\frac{n}{2}\right\} + B\right) \cap A = \emptyset$, as required.
\end{proof}

\begin{figure}[t!]
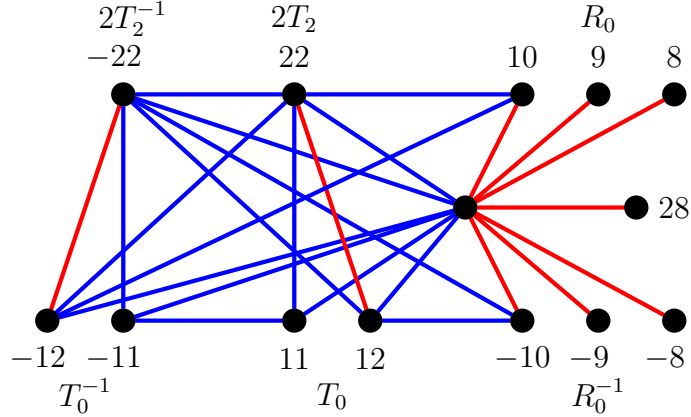

\centering
\ctikzfig{figures/cayley_small_flip_z56_r7_b6}
\caption{Illustration of the closed neighbourhood of $1$ in the Cayley graph construction for $(b, r) = (6, 7)$ and $n = 56$ in the proof of Theorem \ref{newBound}, with the choice of $R_0$, $T_0$ and $T_2$ highlighted.}
\label{cayley_union_counting_fig}
\end{figure}

\begin{proof}[Proof of Theorem \ref{newBound}]
	Let $n \in \mathbb{N}$ such that $n = 8\left(2 + \big\lfloor\frac{r}{2}\big\rfloor + \big\lfloor\frac{b+2}{2}\big\rfloor - 2\big\lfloor\frac{b+2}{6}\big\rfloor\right)$ and consider $\mathbb{Z}_n$, the additive group modulo $n$. By this choice of $n$, the interval $\left(\frac{n}{8}, \frac{n}{4}\right)$ has two disjoint integer intervals $R_0$ and $T_0$ of sizes $\big\lfloor\frac{r}{2}\big\rfloor$ and $\big\lfloor\frac{b+2}{2}\big\rfloor - 2\big\lfloor\frac{b+2}{6}\big\rfloor$ respectively. Choose these intervals such that $\max(R_0) < \min(T_0)$. Now, 
	
	\begin{center}
		$\big\lfloor\frac{b+2}{2}\big\rfloor - 2\big\lfloor\frac{b+2}{6}\big\rfloor = \big\lfloor\frac{b+2}{6}\big\rfloor + \big\lfloor \frac{1}{2}\left(b+2 - 6\big\lfloor\frac{b+2}{6}\big\rfloor\right)\big\rfloor = \big\lfloor\frac{b+2}{6}\big\rfloor + \big\lfloor \frac{(b+2) \!\mod 6}{2}\big\rfloor$ 
	\end{center} and therefore $T_0$ has at least $\big\lfloor\frac{b+2}{6}\big\rfloor$ integers. By our choice of $n$, we can choose $T_0$ such that it has a sub-interval $T_2$ of size $\big\lfloor\frac{b+2}{6}\big\rfloor$ and $\min(T_2) \geq \frac{3n}{16}$.
	
	Define the sets $R_1 = R_0 \ \dot\cup \ R_0^{-1}$ and $T_1 = T_0 \ \dot\cup \ T_0^{-1}$, which are inverse-closed and sum-free. Define $B_1 = T_1 \ \dot\cup \ 2T_2 \ \dot\cup \ 2T_2^{-1}$. Since $T_2$ is an integer interval and $2T_2$ is the sum-set of $T_2$ with itself, then $|2T_2| = 2|T_2| - 1$. Moreover, $|T_2| = |T_2^{-1}|$ and $|T_0| = |T_0^{-1}|$. Therefore, $$|T_1| + |2T_2| + \left|2\left(T_2^{-1}\right)\right| = 2|T_0| + 4|T_2| - 2 = 6\left \lfloor\frac{b+2}{6}\right \rfloor + 2\left \lfloor \frac{(b+2) \!\mod 6}{2}\right \rfloor - 2$$ and hence $|B_1| = b - \left(b \mod 2\right)$.
	
	The sets $R_1$ and $B_1$ have even size, and we may need to add some involutions to them to get the size equal to $r$ and $b$ respectively. Three cases may arise:
	
	\begin{enumerate}[i.]
		\item Both $r$ and $b$ are even and therefore $|R_1| = r$ and $|B_1| = b$. In this case let $\Gamma = \mathbb{Z}_n$, $R = R_1$ and $B = B_1$. In this case $\lambda_{b,r} = 1$. 
		\item Either $r$ is odd or $b$ is odd, in which case we let $\Gamma = \mathbb{Z}_n$. If $r$ is odd we define $R = R_1 \cup \{\frac{n}{2}\}$ and $B = B_1$. Else if $b$ is odd we define $R = R_1$ and $B = B_1 \cup \{\frac{n}{2}\}$. Consequently $R$ and $B$ are inverse-closed and have size $r$ and $b$ respectively. Moreover, $R$ is sum-free. In this case $\lambda_{b,r} = 1$. 
		\item Both $r$ and $b$ are odd, in which case we let $\Gamma = \mathbb{Z}_2 \times \mathbb{Z}_n$. Define $B = \left(\{0\} \times B_1\right) \cup \{(0, \frac{n}{2})\}$ and $R = \left(\{0\} \times R_1\right) \cup \{(1,0)\}$, noting that $(0, \frac{n}{2})$ and $(1, 0)$ are involutions in $\mathbb{Z}_2 \times \mathbb{Z}_n$. Moreover, $R$ is sum-free by choice of $R_1$ and by the properties of the direct product. In this case $\lambda_{b,r} = 2$.
	\end{enumerate}  
 	 
 	Consider $H = \cay{\Gamma}{R}$ coloured monochromatically using colour $2$. Since $R$ is sum-free, then $e_2^H [v] = r = \deg^H_2 (v)$ for any $v$ in $\Gamma$.
 
 	Now consider $H = \cay{\Gamma}{B}$ coloured monochromatically using colour $1$. For any $v$ in $\Gamma$, $\deg_1^G (v) = b$. Moreover, there are at least $2 \left\lfloor\frac{b+2}{6}\right\rfloor^2$ edges in the open neighbourhood of $v$, since the group is Abelian and therefore the $\left\lfloor\frac{b+2}{6}\right\rfloor^2$ sums from $T_2$ to $2T_2$, and from $T_2^{-1}$ to $2T_2^{-1}$, all contribute an edge to the open neighbourhood. Hence, $$e_1^G [v] \geq b + 2\left\lfloor\frac{b+2}{6}\right\rfloor^2 > r = e_2^H [v].$$
 
 	Finally observe that by construction, as a consequence of Lemma \ref{sumsetLemma} and properties of the direct product, $(R + B) \cap R = \emptyset$. By Proposition \ref{cayley_union_counting} (iii), $\cay{\Gamma}{R \cup B}$ is a $(b, r)$-flip graph, and by our choice of $\Gamma$ we obtain the bound on $h(b, r)$.
\end{proof}

It is worth comparing this bound to the one given in Theorem \ref{oldBound} from \cite{caro2023flip}, noting that the new bound offers a significant improvement. This is illustrated in Figure \ref{boundComparison1}, for fixed $b$ and $1 < r - b < 2 \left \lfloor\frac{b}{6}\right\rfloor^2$.

\begin{figure}[h!]
\centering
\includegraphics[width=\textwidth]{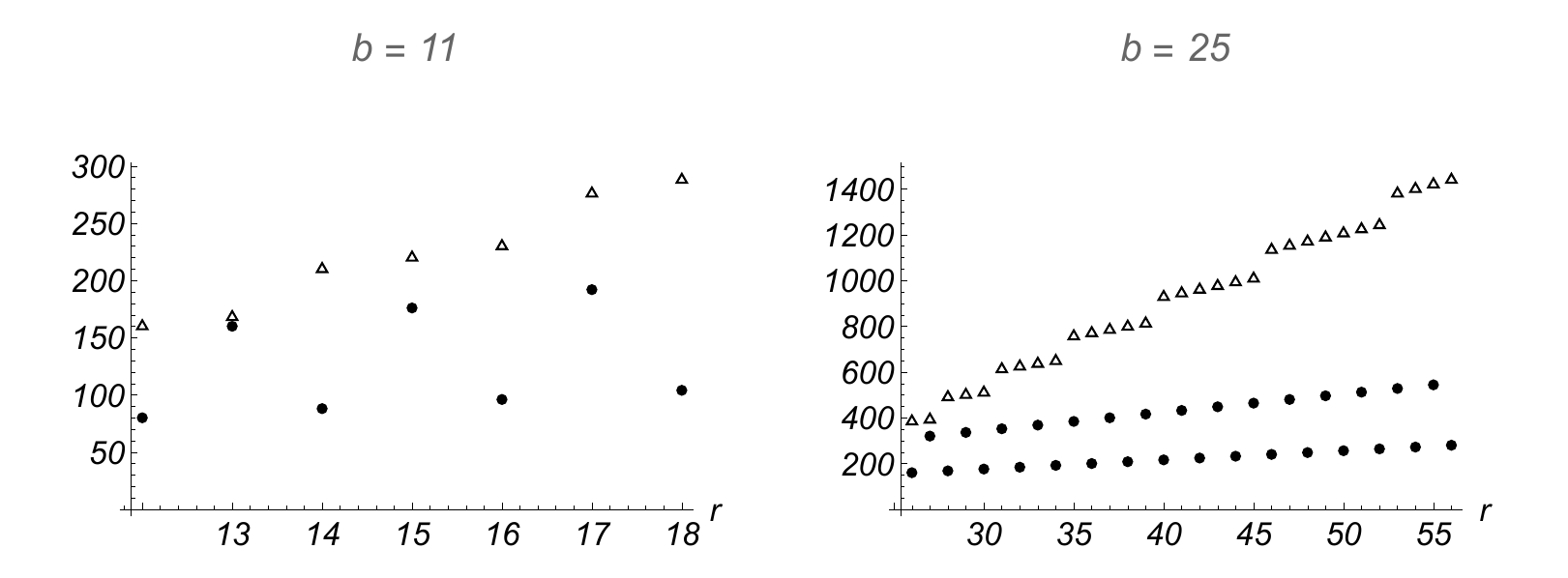}
\caption{Comparison of the bounds in $\triangle$ Theorem \ref{oldBound} and $\bullet$ Theorem \ref{newBound} for $b = 11$ and $b = 25$ over the common range for $r$ between the two bounds.}
\label{boundComparison1}
\end{figure}

Observe, however, that the existing bound in Theorem \ref{oldBound} holds for a wider range of values of $b$, suggesting that further work is to be done towards a unified bound. 
\section{Bounding $q(k)$}

In this section we shall consider the problem in \cite{caro2023flip} on establishing bounds on $q(k)$ for $k > 3$,
\begin{equation} \label{qKBound}
	\max\left\{1, \left\lceil\frac{k}{4}\right\rceil - 1\right\} \leq q(k) < \begin{cases}
 	\ \frac{k}{3} & \mbox{if $k \equiv 0 \ (\!\!\!\!\!\!\mod 3)$} \\
 	\left\lceil \frac{k}{2}\right\rceil & \mbox{otherwise}
 \end{cases}.
\end{equation}

\subsection{Upper bounds on $q(k)$}

We begin by proving the upper bound in (\ref{qKBound}). The following ideas stem from a communication by Caro \cite{caroPersonalComm}.

\begin{theorem}\label{qKUpperBound}
Let $k \in \mathbb{N}, k \geq 2$. Then $q(k) < \begin{cases}
 	\ \frac{k}{3} & \mbox{if $k \equiv 0 \ (\!\!\!\!\mod 3)$} \\
 	\left\lceil \frac{k}{2}\right\rceil & \mbox{otherwise} \\
 \end{cases}$.
\end{theorem}

\begin{proof}
	The cases $k = 2$ and $k = 3$ immediately follows from Theorems \ref{RBClassification} and \ref{3colour_bound}. Hence consider $k \geq 4$. Let $G$ be an $(a_1, \dots, a_k)$-flip graph. Consider the case when $k \equiv 0 \ (\!\!\!\!\mod 3)$ and let $p = \frac{k}{3}$.  Re-colour the edges of $G$ such that the $p$ colours $p(j-1) + 1, \dots, pj$ are coloured $j$ for $j \in \{1, 2, 3\}$. For $j \in \{1, 2, 3\}$, define $b_j = \sum_{i = 1}^{s} a_{(i-1)s + j}$. Note that $b_1 < b_2 < b_3$ by the monotonicity of the $k$-flip sequence and the fact that each $b_j$ is a sum of $s$ terms. 
	
	Applying a similar argument to the coloured closed neighbourhood sizes, one observes that $G$ is a $(b_1, b_2, b_3)$-flip graph.
	
	By Theorem \ref{3colour_bound}, it follows that $b_3 \leq 2 (b_1)^2$ and hence $a_k \leq 2k^2 (a_s)^2$. Therefore $a_k$ is bound in $a_{\frac{k}{3}}$ and hence $q(k) < \frac{k}{3}$ when $k \equiv 0 \ (\!\!\!\!\mod 3)$.
	
	Otherwise, consider $t = \left\lceil\frac{k}{2}\right\rceil$. Let $c_1 = \sum_{i = 1}^t a_i$ and $c_2 = \sum_{i = 1}^{k-t} a_{t+i}$. Two cases are possible: either $2t = k$ or $2t - 1 = k$. In the case that $2t - 1 = k$, it need not necessarily be that case that $c_1 < c_2$. Indeed, suppose that $c_1 \geq c_2$. Then $t a_t \geq a_k$ and hence $a_k$ is bounded in $a_{\left\lceil\frac{k}{2}\right\rceil}$.
	
	Consider the cases when $2t = k$ or $2t - 1 = k$ but $c_1 < c_2$. In both these cases, by a similar argument to the case when $k \equiv 0 \ (\!\!\!\!\mod 3)$, we have that $G$ is a $(c_1, c_2)$-flip graph. By Theorem \ref{RBClassification} it follows that $c_2 < \binom{c_1 + 1}{2}$ and hence $a_k$ is quadratically bound in $a_{\left\lceil\frac{k}{2}\right\rceil}$. It follows that $q(k) < \left\lceil\frac{k}{2}\right\rceil$.
\end{proof}

\subsection{Lower bounds on $q(k)$}

In this section we shall prove the lower bound in (\ref{qKBound}). We first note the following useful result on the existence of flipping intervals which are realised by flip graphs having a difference of 1 in the sequence of coloured closed neighbourhood sizes. This will be a critical step in our proof of the lower bound.

\begin{theorem}\label{flippingIntervalCorr}
Let $q, b \in \mathbb{N}$ such that $q > 1$, $b \geq 101$ and $\left\lfloor \frac{1}{4} \big(b^2 - 10b^{\frac{3}{2}}\big)\right\rfloor \geq q - 1$. Then $[b, b+q-1]$ is a flipping interval realised by a $q$-flip graph $G$ with $e_j [v] - e_{j+1} [v] = 1$ for all $1 \leq j < q$ and $v \in V(G)$. 
\end{theorem}

\begin{proof}
	Follows immediately from the construction in the proof of Theorem 5.2 and Corollary 5.3 of \cite{caro2023flip}.
\end{proof}

We next show that, if $q$-flip graphs satisfying particular properties exist, then $(a_1, \dots, a_k)$-flip graphs exist for $k > 4q$, where $a_k$ can be arbitrarily large whilst $a_i$ is constant for $1 \leq i \leq q$. 

\begin{lemma} \label{GapsProp}
	Let $q, k \in \mathbb{N}$ such that $1 < q < \frac{k}{4}$. Let $D_1, \dots, D_q \in \mathbb{N}$ such that $D_q (k - 4q) > 1 + \xi q(q-1) + 5\binom{k-q}{2}$ where $\xi = \max\limits_{1 \leq j < q}\{D_j - D_{j+1}\}$. 
	
	If there exists a $(a_1, \dots, a_q)$-flip graph $F$ such that for every $v \in V(F)$ and $1 \leq j \leq q$, $e_j^F [v] = D_j$, then given any $N \in \mathbb{N}$ there exists a $(a_1, \dots, a_k)$-flip graph for some $a_{q+1}, \dots, a_k \in \mathbb{N}$ where $a_k > N$.
\end{lemma}

\begin{proof}
Let $\Gamma$ be a sufficiently large finite group such that $\Gamma$ has $k - q - 1$ disjoint and inverse-closed subsets $S_j$ where $|S_j| = k - q - j$ for $1 \leq j < k - q$, and such that $S = \bigcup\limits_{j = 1}^{k - q - 1} S_j$ is sum-free in $\Gamma$. 

Consider $K = \cay{\Gamma}{S}$ with an edge-colouring such that the edges labelled in $S_j$ are coloured using $q + j$. Hence for any $v \in \Gamma$, $\deg_{q+j}^K (v) = k - q - j$ and by the sum-free condition on $S$, $e_{q+j}^K [v] = k - q - j$.

Now consider the coloured Cartesian product $F \ \square \ K$, which is $k - 1$ coloured since $F$ is coloured using $1, \dots, q$ and $K$ is coloured using $q + 1, \dots, k - 1$. 

By Lemma \ref{CPColourLemma}, for a given colour $j$ the graph has colour-degree $a_j$ inherited from $F$ for $1 \leq j \leq q$, and $a_j = k - j$ inherited from $K$ when $q < j < k$. Likewise, the number of edges coloured $j$ in a closed neighbourhood is $D_j$ inherited from $F$ for $1 \leq j \leq q$ and $D_j = k - j$ inherited from $K$ for $q < j < k$. 

Finally, note that $F \ \square \ K$ is $\mu = \binom{k-q}{2} + \sum\limits_{i = 1}^q a_i$ regular. Let $t \in \mathbb{N}$ such that $$t \geq \dfrac{1 + \mu + 2\sum\limits_{i = 1}^k D_i}{(k-q)\min\limits_{q < i < k}\{D_i - D_{i+1}\}}$$ and let $H$ be a $\rho = (k-q)t + \binom{k-q}{2}$ regular bipartite graph. For $1 \leq j \leq k - q$, colour $t + j - 1$ matchings of $H$ using $q + j$. 

Let $G$ be the coloured product $H \boxtimes (F \ \square \ K)$. By Lemma \ref{SPColourLemma}, for $v \in V(G)$ and $1 \leq j \leq k$, the edge-colouring in $G$ is such that $$\deg_j (v) = 
\begin{cases} 
a_j & 1 \leq j \leq q \\
a_j + (t + i - q - 1)(1 + \mu)& q < j \leq k
\end{cases}
$$ and $$e_j [v] = \begin{cases} 
D_j (\rho + 1) & 1 \leq j \leq q \\
D_j (\rho + 1) + (t + j - q - 1)\left(1 + \mu + 2\sum\limits_{i=1}^k D_i\right)& q < j \leq k
\end{cases}.
$$

We will show that $G$ as constructed and edge-coloured is a $k$-flip graph. Observe that for $1 \leq j \leq k$, $a_j < 1 + \mu$. Hence for $j = q$ we have that $$\deg_q (v) = a_q < 1 + \mu < a_{q+1} + t(1 + \mu) = \deg_{q+1} (v)$$ and for $j > q$ we have that $\deg_{j+1} (v) - \deg_j (v) = \mu > 0$. Consequently the colour-degree sequence in $G$ is strictly increasing.

Since for $1 \leq j < q$ we have $D_j > D_{j+1}$ then in $G$ we have $e_j [v] > e_{j+1} [v]$. Next note that $\rho + 1 = (k-q)t\kappa$ for some $\kappa > 1$. Hence $(D_q - D_{q + 1})(\rho + 1) > (D_q - D_{q+1})(k-q)t$. Therefore to show that $e_q [v] > e_{q+1} [v]$ it suffices to show $$(D_q - D_{q+1})(k-q) > 1 + \mu + 2\sum\limits_{i=1}^k D_i.$$ 

From the lower bound on $D_q$ in the theorem statement we have that
\begin{align*}
	&D_q (k - q) &&\\
	&> 1 + (3q)D_q + \xi q(q-1) +  5\binom{k-q}{2} &&\\
	&> 1 + \sum\limits_{i = 1}^q a_i + (2q)D_q + \xi q(q-1) +  5\binom{k-q}{2} && \because a_1 < \dots < a_q \leq D_q\\
	&= 1 + \mu + (2q)D_q + \xi q(q-1) +  4\binom{k-q}{2} && \because \mu = \sum\limits_{i = 1}^q a_i + \binom{k-q}{2}\\
	&\geq 1 + \mu + 2\sum_{i=1}^q D_i + 4\binom{k-q}{2} && \because D_{q-i} \leq D_q + i\xi\\
	&= 1 + \mu + 2\sum_{i=1}^k D_i + 2\binom{k-q}{2} && \because \binom{k-q}{2} = \sum\limits_{i = q + 1}^{k} D_i\\
	&= D_{q+1} (k - q) + 1 + \mu + 2\sum_{i=1}^k D_i && \because D_{q+1} = k - q - 1
\end{align*}
as required and therefore the colours $q$ and $q+1$ flip in $G$.

Consider the final case when $q < j < k$. By the choice of $t$ and $\kappa > 1$, we have that $$(D_j - D_{j + 1})(\rho + 1) = (D_j - D_{j+1})(k-q)(t\kappa) > 1 + \mu + 2\sum\limits_{i = 1}^k D_i$$ which we can re-arrange to get $e_j [v] > e_{j+1} [v]$. Hence the sequence of closed neighbourhood sizes is strictly decreasing as required.

It follows that $G$ is a flip graph on $k$ colours, such that for any vertex $v$,  the difference between $\deg_q (v)$ and $\deg_k (v)$ grows in $t$ as $t \to \infty$.
\end{proof}

We are finally in a position to prove the lower bound in (\ref{qKBound}).

\begin{theorem}\label{qKLowerBound}
	Let $k \in \mathbb{N}$ such that $k > 3$. Then for any $q \in \mathbb{N}$ such that $q = 1$ or $q < \frac{k}{4}$, there exists $a_1, \dots, a_q \in \mathbb{N}$ such that given any $N \in \mathbb{N}$ there exists a $(a_1, \dots, a_k)$-flip graph for some $a_{q+1}, \dots, a_k \in \mathbb{N}$ where $a_k > N$.
\end{theorem}

\begin{proof}
	The case $q = 1$ follows immediately from Theorem \ref{arbitraryGapsThm_h_1}. Hence consider the case when $1 < q < \frac{k}{4}$. 
	
	Let $b \in \mathbb{N}$ be sufficiently large such that $b \geq \max\left\{101, \dfrac{1 + q(q-1)+5\binom{k-q}{2}}{k-4q}\right\}$ and $\left\lfloor \frac{1}{4} \big(b^2 - 10b^{\frac{3}{2}}\big)\right\rfloor \geq q - 1$. By Corollary \ref{flippingIntervalCorr} and the choice of $b$, there exists a $(b, \dots, b + q - 1)$ flip graph $F$ where for every vertex $v \in V(F)$ and $1 \leq j < q$, $\xi = \max\limits_{1 \leq j < q}\{e_j [v] - e_{j+1} [v]\} = 1$. Moreover, $$e_q [v] (k - 4q) \geq (b + q - 1)(k - 4q) > b(k - 4q) \geq 1 + \xi q(q-1)+5\binom{k-q}{2}$$ and therefore the result follows as an immediate consequence of Lemma \ref{GapsProp}.
\end{proof}
\section{Concluding remarks}

In this paper we set out to answer two problems of \cite{caro2023flip} concerning bounds on $h(b,r)$ and $q(k)$ for flip graphs. Given $k > 4$, since $h(a_1, \dots, a_k) > a_k$, as a consequence of Theorem \ref{qKLowerBound} we have that there exist $k$-flip sequences where $a_i$ are fixed for $1 \leq i < \frac{k}{4}$, but $a_k \to \infty$. In particular one observes that consequently $h(a_1, \dots, a_k)$ is not bounded above by a polynomial in $a_i$ for $1 \leq i < \frac{k}{4}$.

This contrasts the cases $k = 2$ and $k = 3$, where $h(a_1, a_2)$ and $h(a_1, a_2, a_3)$ are quadratically bound in $a_1$. In light of this, we pose the following problem.

\begin{problem}
	For $k \geq 4$, is there a smallest integer $p(k)$, $\frac{k}{4} \leq p(k) \leq k$, such that $h(a_1, \dots, a_k)$ is polynomially bound in $a_{p(k)}$?
\end{problem}

Expanding the range of admissible values $r$ for fixed $b$ such that $h(b, r)$ is $\Theta(b + r)$ is another problem of interest, possibly requiring more delicate arguments involving sum-free sets in non-Abelian groups. 

\begin{problem}
For $b, r \in \mathbb{N}$ such that $3 \leq b < r \leq \binom{b + 1}{2} - 1$, what is the largest value of $r$ such that $h(b, r) = \Theta(b + r)$?
\end{problem}

%
%
%
%
%

\section*{Acknowledgements}

The author would like to thank Yair Caro for the re-colouring idea behind Theorem \ref{qKUpperBound}, along with Josef Lauri and Christina Zarb for their input which helped improve the presentation of this paper. 

The author would also like to thank the anonymous reviewers for their valuable input, which also helped improve the presentation of this paper. 

\bibliographystyle{plain}
\bibliography{flip_graphs_bibliography}

\begin{thebibliography}{1}

\bibitem{ABDULLAH20151}
M.A. Abdullah and M.~Draief.
\newblock Global majority consensus by local majority polling on graphs of a given degree sequence.
\newblock {\em Discrete Applied Mathematics}, 180:1--10, 2015.

\bibitem{Alon_Kleitman_1990}
N.~Alon and D.~J. Kleitman.
\newblock Sum-free subsets.
\newblock In A.~Baker, B.~Bollobás, and A.Editors Hajnal, editors, {\em A Tribute to Paul Erd\H{o}s}, page 13–26. Cambridge University Press, 1990.

\bibitem{caroPersonalComm}
Y.~Caro.
\newblock \textit{personal communication}, 2023.

\bibitem{caro2023flip}
Y.~Caro, J.~Lauri, X~Mifsud, R.~Yuster, and C.~Zarb.
\newblock Flip colouring of graphs.
\newblock {\em Graphs and Combinatorics}, 40(106), 2024.

\bibitem{caro2018effect}
Y.~Caro and R.~Yuster.
\newblock The effect of local majority on global majority in connected graphs.
\newblock {\em Graphs and Combinatorics}, 34(6):1469--1487, 2018.

\bibitem{FISHBURN1986165}
P.C. Fishburn, F.K. Hwang, and Hikyu Lee.
\newblock Do local majorities force a global majority?
\newblock {\em Discrete Mathematics}, 61(2):165--179, 1986.

\bibitem{GreenBen2005Ssia}
B.~Green and I.~Z. Ruzsa.
\newblock Sum-free sets in abelian groups.
\newblock {\em Israel Journal of Mathematics}, 147(1):157--188, 2005.

\bibitem{sheffield2025}
N.~S. Sheffield and Z.~Xi.
\newblock Graphs with the same edge count in each neighborhood.
\newblock \textit{arXiv preprint arXiv:2507.14473}, 2025.

\bibitem{TaoVan}
T.~Tao and V.~Vu.
\newblock Sumfree sets in groups: a survey.
\newblock {\em Journal of Combinatorics}, 8, 03 2016.

\end{thebibliography}

\end{document}